\documentclass[12pt]{amsart}

\usepackage{hyperref}
\usepackage{graphicx}
\usepackage[ansinew]{inputenc}
\usepackage{enumerate}
\usepackage{amsmath}
\usepackage{amsfonts}
\usepackage{amssymb}
\usepackage{amsthm}
\usepackage{multirow}
\usepackage{graphicx}
\usepackage{color}
\usepackage{microtype}
\input xy
\xyoption{all}
\usepackage{url}
\usepackage{varioref}
\usepackage{eso-pic}
\usepackage{tikz}
\usepackage{float}
\usepackage{dsfont}
\usepackage{listings}

\lstdefinelanguage{Magma}
{
keywords={for,end,if,then,else,elif,while,function,return,cat,&,and,or },
morekeywords={Seqset,Setseq,Polytope,AutomorphismGroup,RowSequence,IdentifyGroup,
	      Subgroups,PermutationMatrix,Generators,MatrixGroup,Transpose},
sensitive=false,
morecomment=[l]{//},
morecomment=[s]{/*}{*/},
morestring=[b]",
}

\lstnewenvironment{codice_magma}[1][]
{\lstset{basicstyle=\small\ttfamily, columns=fullflexible,
language=Magma,
keywordstyle=\color{red}\bfseries,
commentstyle=\color{blue},
numbers=left, numberstyle=\tiny,
stepnumber=2, numbersep=5pt, frame=single, #1}}{}

\usepackage{multicol}
\newtheorem{thm}{Theorem}
\newtheorem{lem}[thm]{Lemma}
\newtheorem{prop}[thm]{Proposition}

\theoremstyle{definition}

\theoremstyle{remark}

\renewcommand{\H}{H}

\newcommand{\Z}{\mathbb{Z}}
\newcommand{\Q}{\mathbb{Q}}

\newcommand{\F}{\mathbb{F}}
\newcommand{\PP}{\mathbb{P}}
\newcommand{\OO}{\mathcal{O}}

\newcommand{\m}[1]{\mathcal{#1}}

\DeclareMathOperator{\Pic}{Pic}

\DeclareMathOperator{\Bl}{Bl}

\DeclareMathOperator{\Rk}{Rk}
\DeclareMathOperator{\e}{e}

\begin{document}

\title{An Unbounded Family of Log Calabi-Yau Pairs}
\author{Gilberto Bini}
\address[Gilberto Bini]{Department of Mathematics, University of Milan,
Via Cesare Saldini, 50,
I-20133 Milano, Italy}
\email{gilberto.bini@unimi.it}

\author{Filippo F. Favale}
\address[Filippo F. Favale]{Department of Mathematics, University of Trento, via Sommarive 14,
I-38123 Trento, Italy}
\email{filippo.favale@unitn.it}
\date{\today}

\subjclass[2010]{14J30, 14J32, 14J60}

\begin{abstract}
We give an explicit example of log Calabi-Yau pairs that are log canonical and have a linearly decreasing Euler characteristic. This is constructed in terms of a degree two covering of a sequence of blow ups of three dimensional projective bundles over the Segre-Hirzebruch surfaces $\F_n$ for every positive integer $n$ big enough.
\end{abstract}

\maketitle
\textit{Keywords}: log Calabi-Yau pairs, geography of threefolds, projective bundles.

\section{Introduction}

\noindent A log Calabi-Yau pair $(Y,D)$ consists of a proper variety $Y$ and an effective $\Q$-divisor $D$ such that $(Y,D)$ is log canonical and $K_X+D$ is $\Q$-linearly equivalent to zero: see, for instance, \cite{kol}. A Calabi-Yau variety can be viewed as $(Y,0)$. If $Y$ is a Fano variety such that $D$ is $\Q$-linearly equivalent to the anticanonical divisor, then $(Y,D)$ is a log Calabi-Yau pair, provided it is log canonical.

\noindent Let us take into account three dimensional log Calabi-Yau pairs. As well known, there exist finitely many deformation types of Fano threefolds. As a result, there are finitely many possible values for their Euler characteristic. Conjecturally, this should be true for the collection of all Calabi-Yau threefolds too. Here by Calabi-Yau threefold we mean a complex K\"{a}hler compact manifold with trivial canonical bundle and no $p$-holomorphic forms for $p=1,2$. Since general log Calabi-Yau pairs interpolate between these two extremes, it is natural to wonder whether they are bounded or not. In this paper, we prove the following result.

\begin{thm}
\label{mainresult}
There exists an integer $N_0$ such that, for every $n\geq N_0$ there exists a log Calabi-Yau threefolds $(Y, D)$ with the Euler characteristic of $Y$ given by
$$
e(Y)=-48n-46.
$$

\noindent Moreover, $Y$ is smooth and its Kodaira dimension is negative. Additionally, we have $K_{Y}+D=0$, where $D$ is a divisor isomorphic to a $K3$ surface.
\end{thm}

\noindent Recently, Di Cerbo and Svaldi in \cite{DCS} prove that log Calabi-Yau pairs are bounded. One of their assumption is that the pair are klt. Notice that there is no contradiction between their result and ours; indeed, the example in Theorem \ref{mainresult} is not klt but log canonical.
\vspace{2mm}

\noindent The proof of Theorem \ref{mainresult} is constructive. More specifically, we  describe a collection of log Calabi-Yau threefolds with the properties mentioned above. First, take into account the Segre-Hirzebruch surface $\F_n$ for any positive integer $n$. Next, fix a suitable decomposable vector bundle on each $\F_n$, namely 
$$
\m{V}:=\OO_{\F_n} \oplus \OO_{\F_n}(2C_0-F),
$$
where $C_0$ is the unique effective divisor on $\F_n$ such that $C_0^2=-n$ and $F$ is the class of the fiber with respect to the $\PP^1$-bundle structure on $\F_n$.  For any $n$ denote by $X$ the scroll defined as $\PP(\m{V})$, the projective bundle of hyperplanes in $\m{V}$. For futher information about these scrolls, see, for instance, \cite{fanflam}.
\vspace{2mm}

\noindent If the linear system $|-2K_X|$ had a smooth member, then the double covering of $X$ - branched along it - would be a smooth Calabi-Yau manifold. Unfortunately, this is not the case. The base locus of $|-2K_X|$ is given by a smooth rational curve. Luckily, the multiplicity of the generic section along the base locus is three. This requires a careful analysis of the cohomology group $H^0(X,-2K_X)$, which can be carried out more easily for $n$ big enough.
\vspace{2mm}

\noindent If we blow up $X$ along the smooth curve in the base locus of the bianticanonical system, we obtain a smooth threefold $X_1$. The linear system $|-2K_{X_1}|$ is not basepoint free. The base locus is given by a smooth rational curve $\gamma_1$. In order to resolve a generic section of the linear system $|-2K_X|$, we blow up $X_1$ along $\gamma_1$. We obtain a smooth threefold $X_2$. The degree two branched cover $Y_2$ along a smooth section of $-2K_X-2E_1-4E_2$ is not normal. Taking the normalization of it is equivalent to taking the branched covering of $X_2$ along a smooth member of the linear system $-2K_{X_2}- 2E_2=-2K_X-2E_1-4E_2$. 
\vspace{2mm}

\noindent Finally, in order to calculate the Euler characteristic of $Y_2$ for $n$ big enough, it suffices to determine that of $X_2$ and that of a smooth surface in $|-2K_{X_2}-2E_2|$. The former follows from the cohomology of  blow ups along a submanifold and the latter from the Chern classes of it: see, for instance, \cite{GH}.
\vspace{2mm}

\noindent Our construction relies on the choice of the vector bundle $\m{V}=\OO_{\F_n} \oplus \OO_{\F_n}(2C_0-F)$. It is important to stress that this is only one of the possible choices in order to arrive at an unbounded family of log Calabi-Yau pairs. To be more precise, we analysed all the cases as $\m{V}$ varies among the rank $2$ vector bundle on $\F_n$ that are decomposable. Our method yields a double cover, which is a smooth Calabi-Yau threefold, only for a finite number of cases. We expect that for the great majority of the other cases the situation is similar to that presented in this paper: one can mimic the construction and obtain a log Calabi-Yau pair. 
\vspace{2mm}

\noindent The paper is organized as follows. In Section 2 we recall some preliminary results. Section 3 is devoted to describing the bianticanonical system of the scroll $X$, in particular a desingularitazion of a generic section of it. At last, Section 4 concludes the exposition with the computation of the Euler characteristic, thus showing that it is in fact unbounded! 
\vskip 0.7cm
{\bf Acknowledgements.} The authors are both supported by INdAM - GNSAGA and by FIRB2012 "Moduli spaces and applications", which are granted by MIUR. We would like to thank Ciro Ciliberto and Claudio Fontanari for helpful remarks on the problem and the manuscript.


\section{Some Preliminary Results}

\noindent In this section we recall some basic facts and prove some results that will be applied in what follows. For further details, the reader is referred to \cite{Hag}, p. 369 ff. 

\noindent Let $S$ be a smooth projective surface and denote by $X$ the projective bundle associated to a rank $2$ vector bundle $\m{V}$ on $S$. To avoid confusion we recall that $X$ is the projective bundle $\PP(\m{V})$ over the base $S$, where $\PP(\m{V})$ is the projective bundle of hyperplanes in $\m{V}$. In what follows we will set $\tau$ to be $c_1(\OO_X(1))$.

\begin{lem}
\label{PROP:CHERNXANDD}
Denote by $\varphi:X\rightarrow S$ the fibration given by the projective bundle structure. Then the following identities hold:
$$c_1(X)=2\tau+\varphi^*(c_1(S)-c_1(\m{V}));$$
$$c_2(X)=\varphi^*(c_2(S)-c_1(\m{V})c_1(S))+2\varphi^*c_1(S)\tau;$$
$$c_3(X)=2\varphi^*(c_2(S))\tau.$$

\end{lem}

\begin{proof}

\noindent We have the exact sequences
\begin{equation}
\label{SES:Rel1}
0\rightarrow T_{X/S}\rightarrow T_X\rightarrow\varphi^*T_{S}\rightarrow 0,
\end{equation}
\begin{equation}
\label{SES:Rel2}
0\rightarrow \OO_X\rightarrow (\varphi^*\m{V}^{\vee})\otimes \OO_X(1)\rightarrow T_{X/{S}}\rightarrow 0.
\end{equation}
\vspace{2mm}

\noindent Recall also that $H^*(X)$ is generated as an $H^*(S)$-algebra by $\tau$ with the single relation
\begin{equation}
\label{SES:Relaz}
\tau^2-\varphi^*c_1(\m{V})\tau=0.
\end{equation}

We have
\begin{equation*}
c_1((\varphi^*\m{V}^{\vee})\otimes \OO_X(1))=\varphi^*c_1(\m{V}^{\vee})+2\tau=-\varphi^*c_1(\m{V})+2\tau,
\end{equation*}
\begin{equation*}
c_2((\varphi^*\m{V}^{\vee})\otimes \OO_X(1))=\varphi^*c_2(\m{V}^{\vee})+\varphi^*c_1(\m{V}^{\vee})\tau+\tau^2=\varphi^*c_2(\m{V})-\varphi^*c_1(\m{V})\tau+\tau^2.
\end{equation*}

By \eqref{SES:Relaz}, this yields
\begin{equation}
c((\pi^*\m{V}^{\vee})\otimes \OO_X(1))=1+(2\tau-\varphi^*c_1(\m{V})).
\end{equation}

From the exact sequences \eqref{SES:Rel1} and \eqref{SES:Rel2}, we get
\begin{multline}
\label{EQ:CHERNX}
c(X)=c(T_{X/S})\varphi^*c(T_S)=c((\varphi^*\m{V}^{\vee})\otimes \OO_X(1))\varphi^*c(T_S)=\\
=(1+(2\tau-\varphi^*c_1(\m{V})))\varphi^*c(S)=\\
=1+[2\tau-\varphi^*c_1(\m{V})+\varphi^*c_1(S)]+[\varphi^*c_2(S)+\varphi^*c_1(S)(2\tau-\varphi^*c_1(\m{V}))]+[2\varphi^*c_2(S)\tau]=\\
=1+[2\tau+\varphi^*(c_1(S)-c_1(\m{V}))]+[\varphi^*(c_2(S)-c_1(\m{V})c_1(S))+2\varphi^*c_1(S)\tau]+[2\varphi^*c_2(S)\tau]
\end{multline}
\end{proof}
 
\noindent In order to determine the cohomology of line bundles on $X$, we are going to apply the following result. We will recall it here for the sake of completeness: see, for instance, \cite{Hag}, pag. 253, Ex 8.4 (a).
\begin{lem}
\label{LEMMA:HAGPUSH}
Let $\m{V}$ be a vector bundle on a smooth surface $S$. Let $X=\PP(\m{V})$ and define $\tau$ as before. Then
$$\varphi_*\OO_X(a\tau)=0 \qquad \mbox{ if } a<0,$$
$$\varphi_*\OO_X(a\tau)=S^a(\m{V}) \qquad \mbox{ if } a\geq0,$$
$$R^i\varphi_*\OO_X(a\tau)=0 \qquad \forall a\in \Z \mbox{ if } 0<i<\Rk(\m{V})-1 \mbox{ or if } i\geq \Rk(\m{V}),$$
$$R^{\Rk(\m{V})-1}\varphi_*\OO_X(a\tau)=0 \qquad \mbox{ if } a> -\Rk(\m{V}).$$
\end{lem}


\section{A Generic Member of the Bianticanonical Linear System}

\noindent From now onwards, $S$ will be the Segre-Hirzebruch surface $\F_n=\PP(\OO_{\PP^1}\oplus \OO_{\PP^1}(-n))$, with $n$ positive. Recall that $\Pic(\F_n)$ is generated by $C_0$, the only effective divisor on $S$ such that $C_0^2=-n$, and $F$, the class of a fiber of the $\PP^1$-bundle. Hence, without loss of generality, any decomposable vector bundle of rank $2$, up to tensor product with a line bundle, can be written as $\m{V}=\OO_{\F_n}\oplus\OO_{\F_n}(-A)$, where $A=xC_0+yF$ and $x$ is nonnegative. For the sake of convenience, we will denote by the same symbol a divisor on $S$ and its pullback on $X$. We will denote, as before, by $X$ the projective bundle associated to $\m{V}$

\begin{prop}
\label{PROP:FIXLOCUSP2}
Consider a divisor $D=a\tau+G$ on $X$ where $G=bC_0+cF$ is the pullback of a divisor on $S$. Then the following hold:
\begin{itemize}
\item [i)] If $A=xC_0+yF$ with $y\geq 0$ (i.e., if $A$ is effective), $D$ is effective if and only if $a,b,c\geq 0$.
\item [ii)] If $A=xC_0-yF$ with $y>0$, $D$ is effective if and only if $a\geq 0$ and 
$$(b,c)\in \bigcup_{r=0}^a S_r\qquad \mbox{ with } S_r=\{(b,c)\,|\, b,c\geq 0\}+(rx,-ry).$$
\item [iii)] If $A=xC_0-yF$ with $y>0$, the only prime and rigid divisors on $X$ are $\tau, C_0$ and $\tau+A$.  
\end{itemize}
\end{prop}

\begin{proof}
i) $D=a\tau+bC_0+cF$ is effective if and only if $a\geq 0$; else we have $$\varphi_*\OO_X(a\tau+bC_0+cF)=\varphi_*\OO_X(a\tau)\otimes \OO_S(bC_0+cF)=0.$$
Hence, we can assume $a\geq0$. Doing so, we have
$$H^0(\OO_X(D))=\bigoplus_{r=0}^a\H^0(\OO_S(bC_0+cF-rA))\supset H^0(\OO_S(bC_0+cF)).$$
If $b,c\geq 0$ the divisor is effective.

ii) Assume, now, $A=xC_0-yF$ with $y>0$. In this case
$$V_r=\H^0(\OO_S(bC_0+cF-rA))=\H^0(\OO_S((b-rx)C_0+(c+ry)F))$$ and $H^0(\OO_X(D))$ is not zero exactly when at least one of these spaces is not zero. $V_r$ is not zero exactly when $b\geq rx$ and $c\geq -ry$, i.e., when $(b,c)\in S_r$, so the second claim is proved.

\noindent iii) Finally assume $A=xC_0-yF$ with $y>0$, and consider the effective divisor $D=a\tau+bC_0+cF$ with $a,b\geq 0$ and $-ay\leq c\leq 0$. If $D$ is rigid then $(b,c)\in S_r$ for exactly one value or $r$ (with $0\leq r\leq a$). If $(b,c)\in S_r$ we can write $D$ as $a\tau+rA+b'C_0+c'F$ with $0\leq c'<y$. If $r<a$ we can assume $0\leq b'<x$ whereas, if $(b,c)\in S_a$, we can assume $b'\geq 0$. In both cases, the divisor $D=a\tau+rA+b'C_0$ is effective; hence, if $c'>0$, we have $h^0(\OO_X(D))\geq 2$. This shows that we have to look for rigid divisors among the ones of the form
$$D=a\tau+rA+b'C_0,$$ 
where $a\geq 0, 0\leq r< a$ and $0\leq b'<x$ or with $a\geq 0, r=a$ and $b'\geq 0$. It is not difficult to see that, in this case, $h^0(\OO_X(D))=1$, so we always get a rigid divisor. It is also easy to see that every such divisor can be written as a sum 
$$a_1\tau+a_2C_0+a_3(\tau+A),$$
which proves that $\tau,C_0$ and $\tau+A$ are the only rigid prime divisors on $X$ when $A=xC_0-yF$ and $y>0$.
\end{proof}

We will be interested in the case $\m{V}=\OO_{\F_n}\oplus\OO_{\F_n}(-A)$ with $A=2C_0-F$. Recall that, in this case,
$$
K_X=-2\tau-2C_0-(n+2)F-2C_0+F=-2\tau - 4C_0 - (n + 1)F,
$$
so we have
$$
-2K_X=4\tau+8C_0+(2n+2)F.
$$

\noindent As we will see, if $n$ is big enough, the linear system $|-2K_X|$ does not have smooth members. Thus, we need to describe more closely the base locus and the type of singularities. 

\begin{prop} The base locus of the bianticanonical linear series is given by the complete intersection $\sigma$ of the rigid divisors with class $\tau+A$ and $C_0$.
\end{prop}

\begin{proof} Since $2C_0-F$ is not effective, by Proposition \ref{PROP:FIXLOCUSP2} there are three rigid prime divisors, namely $\tau, C_0$ and $\tau + A$. The intersections of these three divisors are
$$
\tau(\tau+A)=0, \qquad \tau C_0 := \gamma, \qquad (\tau+A)C_0= (\tau- (2n+1)F)_{|_{C_0}}:=\sigma.
$$

The unique surface $T$ with class given by $\tau$ is a Segre-Hirzebruch surface $\F_n$ with standard generators for $\Pic(\tau)$ given by 
$$C_0|_T=\gamma_T, \qquad F|_T=f_T.$$
By standard generators, we mean a basis of effective prime divisors under which the intersection product has representative matrix 
$$\begin{bmatrix}
-a & 1 \\
1 & 0
\end{bmatrix}$$ where $a$ is the (positive) index of the Segre-Hirzebruch surface. In particular, the class of the curve $\gamma$ seen in $T$ is given by $\gamma_T$.

\noindent Denote by $R$ the only surface whose class is $\tau+A$. One can easily see that $R$ is again a Segre-Hirzebruch surface $\F_n$ if one considers the vector bundle $\m{V}'=\m{V}\otimes \OO_S(A)$ and uses the identification 
$$X=\PP(\m{V})=\PP(\OO_S\oplus\OO_S(-A))=\PP(\OO_S\oplus\OO_S(A))=\PP(\m{V}').$$
Indeed the class of $c_1(\OO_{\PP(\m{V}')}(1))=\tau'$ is $\tau+A$ under this identification. The standard generators for $\Pic(R)$ are 
$$C_0|_R=\gamma_R, \qquad F|_R=f_R.$$

\noindent The surface $U$, whose class is $C_0$, is also a Segre- Hirzebruch surface $\F_m$ with $m=2n+1$. The standard generators for the Picard lattice are 
$$(\tau+A)|_U=\gamma_U, \qquad F|_U=f_U.$$

\noindent Notice that 
$$-2K_X=4(\tau+A)+(2n+6)F,$$
so, an eventual base point of $|-2K_X|$ cannot lie outside the surface $R$. In fact, $(2n+6)F$ is globally generated. It is easy to prove that  $\tau+A$ is not a component of $|-2K_X|$, so the base locus of the bianticanonical linear series is contained in $R$. In fact, let us restrict $-2K_X$ to $R$. This yields
$$
(4\tau+8C_0+(2n+2)F)|_R = 8\sigma+(2n+2)f_R,
$$
which shows that $\sigma$ is contained in the base locus of the bianticanonical linear series. Conversely, given a point in such a base locus, it must belong to $\sigma$ because it is in $R$ and nowhere else than in $\sigma$ because $|f_R|$ is globally generated in $R$. Therefore the claim is proved.
\end{proof}

{\bf Remark.} \noindent The curves $\sigma$ and $\gamma$ are the intersection of $C_0$ with $\tau$ and $\tau+A$, respectively. We can also see them inside these surfaces and the following table describes their classes (a "-" simply means that the curve cannot be seen in that particular surface).
\begin{center}
\begin{tabular}{c|c|c|c|c}
 				   & T ($\tau$) & R ($\tau+A$) & U ($C_0$)\\ \hline
$T\cap U = \gamma$ & $\gamma_T$ & - & $\gamma_U+(2n+1)f_U$\\
$R\cap U = \sigma$ & - & $\gamma_R$ & $\gamma_U$
\end{tabular}
\end{center}
From this description (as well as from adjunction) one can see that both $\gamma$ and $\sigma$ are smooth curves of genus $0$. Moreover, $\sigma$ is rigid in both $R$ and $U$, whereas $\gamma$ is rigid only in $T$.

\begin{prop}
\label{PROP:MULT3}
The generic member of the bianticanonical system has multiplicity $3$ along the base locus.
\end{prop}

\begin{proof}
\noindent Define $t,u$ and $r$ to be the sections (uniquely determined up to scalar) such that $$H^0(\OO_X(\tau))=\langle t\rangle \qquad H^0(\OO_X(C_0))=\langle u\rangle \qquad H^0(\OO_X(\tau+A))=\langle r\rangle$$
so the zero loci of $t,u$ and $r$ describe $T,U$ and $R$, respectively. Define $D_i$ to be $-2K_X-i(\tau+A)$. Therefore, there exists a positive integer $N_0$ big enough such that for $n \geq N_0$ the following hold:
\begin{center}
\begin{tabular}{c|c}
$h^0(D_0) = 14n+61 $ & -\\
$h^0(D_1) = 11n+52$ & $h^0(\OO_X(D_0))-h^0(\OO_X(D_1)) = 3n+9$\\
$h^0(D_2) = 8n+40$ & $h^0(\OO_X(D_1))-h^0(\OO_X(D_2)) = 3n+12$\\
$h^0(D_3) = 5n+25$ & $h^0(\OO_X(D_2))-h^0(\OO_X(D_3)) = 3n+15$\\
$h^0(D_4) = 2n+7$ & $h^0(\OO_X(D_3))-h^0(\OO_X(D_4)) = 3n+18$
\end{tabular}
\end{center}
Let us now describe the sections of $\OO_{X}(-2K_X))=\OO_X(4(\tau+A)+(2n+6)F)$. 
\vspace{2mm}

\noindent We have the exact sequence 
\begin{equation}
\label{EQ:EXS1}
\xymatrix@R=1pc@C=2pc{
0\ar@{->}[r] & H^0(\OO_X(D_1))\ar@{^{(}->}[r]^-{-\otimes r} & H^0(\OO_X(-2K_X)) \ar[r] & H^0(\OO_R(-2K_X)) 
}
\end{equation}
Notice that $-2K_X|_R = (8C_0+(2n+2)F)|_R=8\gamma_R+(2n+2)f_R$ hence, as $R=\F_n$, we have
$$h^0(\OO_{R}(-2K_X))=h^0(\OO_{\F_n}(8\gamma_R+(2n+2)f_R))=3n+9=h^0(\OO_X(D_0))-h^0(\OO_X(D_1)).$$

Thus, the restriction map $H^0(\OO_X(-2K_X))\rightarrow H^0(\OO_R(-2K_X))$ in Equation \ref{EQ:EXS1} is indeed surjective. Denote by $V_0$ a subspace of $H^0(\OO_X(-2K_X))$ such that
$$V_0\oplus (H^0(\OO_X(D_1))\otimes \langle r\rangle) \simeq H^0(\OO_{R}(-2K_X)).$$
If $s\in H^0(\OO_X(-2K_X))$, we have a decomposition of $s$ as
$$s=r\alpha_0+\beta_0$$
with $\alpha_0\in H^0(\OO_X(D_1))$ and $\beta_0\in V_0$. In particular, $\beta_0$ does not vanish identically on $R$ (it vanishes on $\gamma_R$ and some $\PP^1$'s transversal to $\gamma_R$).
We can iterate this process by restricting $\alpha_0$ on $R$. As before, we have the following exact sequences, namely:
\begin{equation}
\label{EQ:EXS2}
\xymatrix@R=1pc@C=2pc{
0\ar@{->}[r] & H^0(\OO_X(D_2))\ar@{^{(}->}[r]^-{-\otimes r} & H^0(\OO_X(D_1)) \ar[r] & H^0(\OO_R(D_1))\ar[r] & 0\\
0\ar@{->}[r] & H^0(\OO_X(D_3))\ar@{^{(}->}[r]^-{-\otimes r} & H^0(\OO_X(D_2)) \ar[r] & H^0(\OO_R(D_2))\ar[r] & 0 \\
0\ar@{->}[r] & H^0(\OO_X(D_4))\ar@{^{(}->}[r]^-{-\otimes r} & H^0(\OO_X(D_3)) \ar[r] & H^0(\OO_R(D_3))\ar[r] & 0,
}
\end{equation}
where the surjectivity follows as before by inspecting the dimension of 
$$H^0(\OO_R(D_i))=H^0(\OO_{\F_n}((8-2i)\gamma_R+(2n+2+i)f_R)$$
and observing that it equals $h^0(\OO_X(D_i))-h^0(\OO_X(D_{i+1}))$.
Then we can create the vector spaces $V_i$ such that
$$V_i\oplus (H^0(\OO_X(D_{i+1}))\otimes \langle r\rangle) \simeq H^0(\OO_{R}(D_i))$$
and sections $\alpha_i\in H^0(\OO_X(D_{i+1}))$, $\beta_i\in V_i$ such that
$$\alpha_i=r\alpha_{i+1}+\beta_{i+1}.$$
Finally, the section $s$ has the following form:
\begin{equation}
s=r^4\alpha_3+r^3\beta_3+r^2\beta_2+r\beta_1+\beta_0.
\end{equation}
Notice that $D_0|_R,D_1|_R$ and $D_2|_R$ are divisors with $\sigma_R$ as fixed components so $\beta_i$ for $i=0,1,2$ will vanish on it (with multiplicity greater than or equal to $4$). But the same is not true for $D_3|_R$, which is very ample. In particular, $\beta_3$ can be chosen such that $\beta_3|_R$ vanishes at exactly $5$ points of $\sigma_R$ (this is equal to $\sigma_R\cdot D_3|_R$) which are free on $\sigma_R$ and whose associated curve cut $\sigma_R$ transversely at such points.

In particular, the generic element of $|-2K_X|$ has $\sigma$ as base curve and the multiplicity of $\sigma$ along the generic bianticanonical divisor is $3$.
\end{proof}

\section{Blowing up the Projective Bundle} 

\noindent In this section we will describe a resolution of a generic member of the linear system $|-2K_X|$.
\vspace{2mm}

\noindent Near a point $P$ of $\sigma$ we can choose local coordinates $(x,y,z)$ such that $x=y=0$ is the local equation of $\sigma$ near $P$, $x=0$ and $y=0$ are the local equations of $R$ and $U$ respectively and $z$ is a coordinate on $\sigma$. We can also use $(y,z)$ as local coordinates on $R$. We write, locally
$$s=x^3f+x^4g+x^2y^4f_1+xy^6f_2+y^8f_3,$$
where $f$ is the local expression for $\beta_3$ and $g$ is the local expression for $\alpha_3$.
We can blow up $\sigma$ in $X$ and take the strict transform $\tilde{D}$ of $D:=\{s=0\}$. Near $P$ the blow up $X_1$ looks like
$$\{(x,y,z)\times (l_0:l_1)\,|\, xl_1-yl_0=0 \}.$$
In the local chart $U_0=\{l_0\neq 0\}$ we have coordinates $(x,z,l_1)$ with $y=xl_1$ and the local equation for the exceptional divisor $E$ which is $x=0$. The total transform of $D$ has equation
$$x^3(f+xg+x^3l_1^4f_1+x^4l_1^6f_2+x^5l_1^8f_3),$$
so 
$$\overline{s}=f+xg+x^3l_1^4f_1+x^4l_1^6f_2+x^5l_1^8f_3$$
is a local equazion for $\tilde{D}$. Notice that $f(0,y,z)$ is not identically zero because $f$ is the local expression of $\beta_3$. From the proof of Proposition \ref{PROP:MULT3} we have also that $\tilde{D}$ is smooth along $\sigma$ and hence everywhere (since it is the strict transform of something that has base locus $\sigma$). Unfortunately, $\tilde{D}$ is not a bianticanonical divisors on $X_1$: the bianticanonical class is indeed $\tilde{D}+E$ so we can take a bianticanonical divisor on $X_1$ to be the union of $\tilde{D}$ and $E$. This is reduced, reducible and singular exactly along the intersection $E_1\cdot \tilde{D}$. 

\begin{lem} The divisor $E$ is a Segre-Hirzebruch variety $\F_{n+1}$.
\end{lem}

\begin{proof} The curve $\sigma$ is a complete intersection. More precisely, it is the intersection of the two rigid divisors $C_0$ and $\tau+A$. Thus, the normal bundle is $\OO_{\sigma}(C_0) \oplus \OO_{\sigma}(\tau +A)$. By direct computation, this is isomorphic to $\OO_{\PP^1}(-n) \oplus \OO_{\PP^1}(-2n-1)$, which proves the claim.
\end{proof}

The Picard group of $X_1$ is generated by $\tau$, $C_0$, $F$ and the exceptional divisor $E_1$. By construction, the restriction of $\tau$ to $E_1$ is zero. Moreover, the restriction of $C_0$ to the exceptional divisor is an integer multiple of $f_1=F|_{E_1}$, the class of a fiber of $E_1$ seen as Segre-Hirzebruch surface. This follows from the intersection numbers that are calculated in the next section. Therefore, the Picard group of the exceptional divisor is generated by the restriction of $E_1$ and $F$, respectively. It is not difficult to check that the unique divisor $\gamma_1$ on $E_1$ such that $\gamma_1^2=-n-1$ is given by
$$
\gamma_1=-{E_1}_{|{E_1}}-(2n+1)F_{|E_1}.
$$

The strict transform of the divisor $-2K_X$ is equal to $-2K_X-3E_1$. Its intersection with $E_1$ is given by $3\gamma_1 + 5f_1$. This is an effective divisor on $E_1$, which is made up of the unique curve of self-intersection $-n-1$ and $5$ disjoint fibers. Since we have
$$
-2K_{X_1}=-2K_X-2E_1= (-2K_X-3E_1)+E_1,
$$
the sections of the bianticanonical divisor $-2K_{X_1}$ pass through the curve $\gamma_1$, which is the complete intersection of  $\tau+A-E_1$ (strict transform of the divisor $\tau+A$) and $E_1$.
\vspace{2mm}

\noindent Therefore, we blow up $X_1$ along the curve $\gamma_1$ and obtain a new variety $X_2$ with exceptional divisor $E_2$. To determine its structure, we compute the normal bundle of
$\gamma_1$ which is given by
$$
N_{\gamma_1/X_1}=\OO_{\gamma_1}(E_1) \oplus \OO_{\gamma_1}(\tau +A - E_1)\simeq \OO_{\PP^1}(-n-1) \oplus \OO_{\PP^1}(-n).
$$

Therefore, the exceptional divisor $E_2$ is isomorphic to $\F_1$. Let us denote by $\gamma_2$ and $f_2$ the generators of $E_2$ such that $\gamma_2^2=-1$, $\gamma_2f_2=1$ and $f_2^2=0$. As in the case of $E_1$, we can take $f_2$ to be the restriction of $F$ to $E_2$. As for the other divisor, it is easy to check that 
\begin{equation}
\gamma_2= -{E_2}_{|E_2}-(n+1)F_{|E_2}.
\end{equation}

The bianticanonical divisor of $X_2$ is thus given by
$$
-2K_{X_2}=-2K_{X}-2E_1-2E_2= (-2K_X - 2E_1 - 4E_2)+2E_2.
$$

Let us compute the restriction of the divisor $(-2K_X - 2E_1 - 4E_2)$ to $E_2$. An easy calculation shows that it is equal to $4\rho_2+6f_2$, which corresponds to the class of a smooth irreducible curve on $E_2 \simeq \F_1$. Thus, there is a smooth member of the linear system 
$$
2(-K_{X_2}-E_2)=-2K_X - 2E_1 - 4E_2.
$$

Being $2(-K_{X_2}-E_2)$ even, we can consider the cyclic covering $\beta:Y_2\rightarrow X_2$ of degree two with branch along a smooth member of $-2K_{X_2}-2E_2= 2K_X - 2E_1 - 4E_2$. 
\begin{lem}
$Y_2$ is a smooth threefold and $\beta^*E_2$ is a $K3$ surface. Moreover
the pair $(Y_2,\beta^*E_2)$ is a log Calabi-Yau.
\end{lem}
\begin{proof}
$Y_2$ is clearly smooth as the branch divisor has been chosen to be smooth. Moreover, by \cite{BPHV} pag. 55, we have also
\begin{equation}
\label{EQ:KY2}
K_{Y_2}=\beta^*(K_{X_2}+B_2/2)=-\beta^*(E_2)
\end{equation}
so $(Y_2,\beta^*E_2)$ is a log Calabi-Yau.  Notice that $\beta^*(E_2)$ is a degree two covering of the Segre-Hirzebruch surface $\F_1$ branched along the intersection of $E_2$ with the branch divisor of the covering $\beta:Y_2\rightarrow X_2$. We have already seen that this intersection can be written as the smooth curve $B_2=4\rho_2+6f_2$ on $E_2\simeq \F_1$, i.e. it is a smooth bianticanonical curve on $\F_1$. This is enough to conclude that the canonical divisor of $\beta^*E_2$ is trivial. The Euler characteristic is of $\beta^*E_2$ can be calculated as $2e(E_2) - e(R_2)$, where $R_2$ is the ramification divisor of the restriction of $\beta$ to $\beta^*(E_2)$. Since $\beta$ is a degree two covering, the divisor $R_2$ is isomorphic to the branch divisor $B_2$. This is a curve of genus $9$, so the Euler characteristic of $R_2$ is $-16$. We have hence $e(\beta^*E_2)=24$ so we can conclude that $\beta^*(E_2)$ is a $K3$ surface.
\end{proof}

In the next section, we are going to calculate the Euler characteristic of $Y_2$ for every $n \geq N_0$. To conclude this section, let us prove the following result.

\begin{thm}
Let $Y_2$ be as above. Then we have:
$$
h^{1,0}(Y_2)=0, \qquad h^{2,0}(Y_2)=0, \qquad h^{3,0}(Y_2)=0.
$$
Moreover, $Y_2$ has negative Kodaira dimension.
\end{thm}

\begin{proof}
We need to determine $h^{q,0}(Y_2)$ for $q\geq 1$. Recall that
$$\beta_*\OO_{Y_2}\simeq \OO_{X_2}\oplus\OO_{X_2}(-B_2/2)=\OO_{X_2}\oplus\OO_{X_2}(K_{X_2}+E_2)$$
and that $R^q\beta_*\mathcal{F}=0$ for all $\mathcal{F}$ coherent on $Y_2$ and for all $q\geq 1$. Hence, by Leray spectral sequence, we have
$$
H^q(\OO_{Y_2}) \simeq H^q(\OO_{X_2}) \oplus H^q(\OO_{X_2}(K_{X_2}+E_2)).
$$
$X_2$ is birational to $X$, which is a projective bundle over $\F_n$ so the Hodge numbers $h^{q,0}(X_2)=h^{q,0}(X)$ are zero for $q\geq 1$. Hence we need to prove that $h^q(\OO_{X_2}(K_{X_2}+E_2))$ is zero for $q\geq 1$ in order to conclude the proof.
If $q=3$ this is straightforward: we have 
\begin{equation}
\label{EQ:SOMECOH1}
h^3(\OO_{X_2}(K_{X_2}+E_2))=h^0(\OO_{X_2}(-E_2))=0
\end{equation}
because $E_2$ is effective. We have $h^p(\OO_{X_2}(K_{X_2}))=h^{3-p}(\OO_{X_2})=h^{3-p}(\OO_{X})$ so, 
\begin{equation}
\label{EQ:SOMECOH2}
h^1(\OO_{X_2}(K_{X_2}))=h^2(\OO_{X_2}(K_{X_2}))=0\qquad\mbox{ and }\qquad h^3(\OO_{X_2}(K_{X_2}))=1.\end{equation}
To compute $H^q(K_{X_2}+E_2)$ for $q=1,2$, let us consider the exact sequence
$$
0 \rightarrow \OO_{X_2}(K_{X_2}) \rightarrow \OO_{X_2}(K_{X_2}+E_2) \rightarrow \OO_{E_2}(K_{X_2}+E_2) \rightarrow 0,
$$
which yields, using also Equations \ref{EQ:SOMECOH1} and \ref{EQ:SOMECOH2}, the exact sequences
\begin{equation}
\label{EQ:EXSEQ1}
0 \rightarrow  H^1(\OO_{X_2}(K_{X_2}+E_2)) \rightarrow H^1(\OO_{E_2}(K_{X_2}+E_2)) \rightarrow 0 
\end{equation}\vspace{-12mm}

\begin{equation}
\label{EQ:EXSEQ2}
0 \rightarrow  H^2(\OO_{X_2}(K_{X_2}+E_2)) \rightarrow H^2(\OO_{E_2}(K_{X_2}+E_2)) \rightarrow H^3(\OO_{X_2}(K_{X_2}))\rightarrow 0
\end{equation}

By adjunction, $\OO_{E_2}(K_{X_2}+E_2)$ is the canonical divisor of $K_{E_2}$ so $H^1(\OO_{E_2}(K_{E_2}))=H^{1,2}(\F_1)=0$ (or, alternatively, by Lemma 2.9 of \cite{CMR}). Hence, from the exact sequence \ref{EQ:EXSEQ1}, also $H^1(X_2, \OO_{X_2}(K_{X_2}+E_2))$ is zero.
\vspace{2mm}

\noindent Both the second and the third term of the exact sequence \ref{EQ:EXSEQ2} have dimension $1$ so $h^2(\OO_{X_2}(K_{X_2}+E_2))=0$.
\vspace{2mm}

\noindent In order to see that the Kodaira dimension is $-\infty$, it is enough to observe that $-K_{Y_2}$ is effective and this follows from Equation \ref{EQ:KY2}.
\end{proof}

\section{The Euler Characteristic}

\noindent In this section, we will calculate the Chern numbers of $X_2$. Recall that $X=\PP(\m{V})$ with $\m{V}=\OO_S\oplus\OO_S(-A)$ and $A=2C_0-F$. If $X_1=\Bl_{\sigma}X$, where $\sigma$ is the rational curve cut out by $R$ and $U$. If $E_1$ is the class of the exceptional divisor, we can consider the complete intersection curve cut out by the two divisors $\tau +A - E_1$ and $E_1$. As for the notation, denote by $E_2$ the exceptional divisor of the second blow up.
\vspace{2mm}
We will apply the following lemma:

\begin{lem}
\label{LEM:BLOW}
Let $Z$ be a smooth complex threefold and let 
$$\xymatrix{
C \ar@{^{(}->}[r]^{j} & Z,
}$$
where $C$ is a smooth curve. If $Z'=\Bl_C(Z)$ with exceptional divisor $E$ and blow up map $\pi:Z'\rightarrow Z$. Then the following hold:
\begin{align}
\label{EQ:CHHBL}
c_1(Z')=&\pi^*c_1(Z)-E\\
c_2(Z')=&\pi^*(c_2(Z)-\eta_C)-\pi^*c_1(Z)E\\
H^*(Z')=&H^*(Z)\oplus H^*(E)/H^*(C),
\end{align}
where $\eta_C$ is the class of $C$ in $H^4(Z)$. Moreover, if $\alpha_p\in CH^p(Z)$ and $p+q=3$ with $q\geq 1$, then
\begin{equation}
\label{EQ:INBL}
E\cdot(\pi^*\alpha_2)=0\qquad E^2\cdot(\pi^*\alpha_1)=-j^*\alpha_1\qquad E^3=-c_1(N_{C/Z}).
\end{equation}
\end{lem}

\begin{proof}
The first two identities can be found in \cite{GH}, p. 609. Let $\alpha_p$ be a class in $CH^p(Z)$ and consider the following commutative diagram
$$\xymatrix{
E \ar@{^{(}->}[r]^{\iota}\ar[d]_\pi & Bl_C(Z) \ar[d]^{\pi} \\
C \ar@{^{(}->}[r]_{j}& Z
}$$
If we assume that $q\geq 1$ we can write $E^q=E^{q-1}\cdot E=E^{q-1}\iota_*(1)$ so that
\begin{multline*}
E^q\cdot (\pi^*\alpha_p)=(E^{q-1}\pi^*\alpha_p)\iota_*(1)=
\iota^*(E^{q-1}\pi^*\alpha_p)\cdot 1=\iota^*(E^{q-1})(\pi\circ \iota)^*\alpha_p=\\
=\iota^*(E^{q-1})(j\circ \pi)^*\alpha_p=\iota^*(E^{q-1})\pi^*(j^*\alpha_p)=\pi_*(\iota^*E)^{q-1}\cdot (j^*\alpha_p).
\end{multline*}
The restriction of the exceptional divisor to itself is the tautological class of $E$ when seen as the total space of the projective bundle $\PP(N_{C/Z})\rightarrow C$. If we denote by $h=c_1(\OO_{\PP(N_{C/Z})}(1))$, we have $\iota^*(E)^{q-1}=(-1)^{q-1}h^{q-1}$. By definition we have also $\pi_*(h^{q-1})=s_{q-2}(N_{C/Z})$, where $s_{n}(N_{C/Z})$ is the Segre class of level $n$ of the vector bundle $N_{C/Z}$. To conclude, it is enough to observe that $s_1(N_{C/Z})=-c_1(N_{C/Z})$ and  $s_0(N_{C/Z})=1$. 
\end{proof}

\noindent Recall that
\begin{align}
c_1(X)=&2\tau + 4C_0 + (n + 1)F,\\
c_2(X)=&4\tau C_0 + (2n + 4)\tau F + (-2n + 6)C_0F,\\
c_3(X)=&8\tau C_0 F
\end{align}
and that $\sigma$, the center of the first blow up, is the complete intersection of $\tau+A$ and $C_0$. Hence 
$$N_{\sigma/X}=\OO_\sigma(\tau+A)\oplus\OO_{\sigma}(C_0)$$
and the class $\eta_{\sigma}$ of $\sigma$ in $H^4(X)$ is simply the class of $(\tau+A)C_0$. In order to simplify notation, we will write $\alpha$ to indicate both a class in $X$ and its pullback to $X_1$ and $X_2$. The first Chern class of $X_1$ is simply given by $c_1(X)-E_1$ whereas 
$$c_2(X_1)=c_2(X)-(\tau+A)C_0-c_1(X)E_1.$$

We are blowing up a smooth rational curve so $$E_1=\PP(N_{\sigma/X})=\PP(\OO_\sigma(\tau+A)\oplus\OO_{\sigma}(C_0))$$
is the Segre-Hirzebruch surface $\F_{n+1}$. By \eqref{EQ:CHHBL}, we obtain that the Hodge structure of $X_1$ is pure and $h^{1,1}(X_1)=4$. To recap, we have
\begin{align}
c_1(X_1)=&c_1(X)-E_1,\\
c_2(X_1)=&c_2(X)-(\tau+A)C_0-c_1(X)E_1,\\
c_3(X_1)=&10\tau C_0 F.
\end{align}

Moreover, the relations that characterize the intersection theory on $X_1$ are (here we don't report the ones coming from $X$) given by
$$E_1\tau =0\qquad E_1C_0F=0\qquad E_1^2C_0=n\qquad E_1^2F=-1\qquad E_1^3=3n+1.$$
The first relation follows simply by observing that $\tau$ and $\sigma$ are disjoint so $\tau$ and $E_1$ do not intersect. The others follow from Lemma \ref{LEM:BLOW} using 
$$C_0 j^*\sigma=C_0^2(\tau+A)=-n\qquad F j^*\sigma=C_0(\tau+A)F=1.$$
and 
$$c_1(N_{\sigma/X})=C_0^2(\tau+A)+C_0(\tau+A)^2=-(3n+1).$$

The curve $\gamma_1$ is smooth and rational. If we blow it up, we obtain an exceptional divisor $E_2$ that is isomorphic to $\F_1$. Indeed, the normal bundle of such a curve is isomorphic to $\OO_{\PP^1}(-n-1) \oplus \OO_{\PP^1}(-n)$. Using the same argument as before, we have
\begin{align}
c_1(X_2)=&c_1(X_1)-E_2,\\
c_2(X_2)=&c_2(X_1)-(\tau+A-E_1)E_1-c_1(X_1)E_2,\\
c_3(X_2)=&12\tau C_0 F.
\end{align}
Continuing as before we get
$$E_2\tau=0\qquad E_2C_0 F=0\qquad E_2E_1C_0=0\qquad E_2E_1 F=0\qquad E_2E_1^2=0$$
$$E_2^2C_0=n\qquad E_2^2F=-1\qquad E_2^2E_1 =n\qquad E_2^3=2n+1$$

This is all we need to prove the following theorem
\begin{thm}
For every positive integer $n$ big enough there exists a pair $(Y,D)$ such that 
\begin{itemize}
\item $Y$ is a smooth threefold of negative Kodaira dimension with $$e(Y)-48n-46\qquad \mbox{ and }\qquad h^{q,0}(Y)=0 \ \mbox{ for } q \geq 1;$$
\item $D$ is a smooth K3 surface;
\item $(Y,D)$ is a log canonical log Calabi-Yau pair;
\end{itemize}
\end{thm}

\begin{proof}
Fix $n \geq N_0$ and consider the projective bundle $X=\PP(\m{V})$ over $\F_n$, where $$\m{V} = \OO_{\F_n} \oplus \OO_{\F_n}(-2C_0+F).$$ First, blow up the projective bundle along the base locus of the bianticanonical divisor obtaining $X_1$. Next, blow up such a variety along the base locus of the bianticanonical divisor to obtain $X_2$. Take the degree two covering of $X_2$ with branch $B_2$ as described in the previous sections to finally obtain $Y_2$. Then one can take $Y=Y_2$ and $D=\beta^*E_2$. Everything, apart form the calculation for the Euler characteristics, have been done in the previous sections.
\vspace{2mm}

\noindent In order to compute the Euler characteristic, recall that if $D$ is a smooth irreducible divisor on $X_2$, we have
\begin{equation}
c_2(D)=c_2(X_2)-c_1(X_2)D+D^2
\end{equation}
so
\begin{equation}
\e(D)=(c_2(X_2)-c_1(X_2)D+D^2)N_{D/X_2}.
\end{equation}
Hence, being the branch locus $B_2$ a smooth element of $|-2K_{X_2}-2E_2|$, we have that
$$e(B_2)=48n + 70.$$
The Euler number of $X_2$ is given by $12$ so we have, finally, \begin{equation}
\e(Y_2)=2\e(X_2)-\e(B_2)=2\cdot 12-(48n + 70)=-48n - 46.
\end{equation}
Although feasible by hands, we have done the last computation using \verb|Magma|\footnote{http://magma.maths.usyd.edu.au/}.
\end{proof}

\end{document}